\documentclass{amsart}
\usepackage{amsfonts,amssymb,amscd,amsmath,enumerate,verbatim,calc,graphicx}
\newtheorem{theorem}{Theorem}[section]
\newtheorem{lemma}[theorem]{Lemma}

\newtheorem{corollary}[theorem]{Corollary}
\theoremstyle{definition}
\theoremstyle{definitions}
\newtheorem{definition}[theorem]{Definition}

\newtheorem{remark}[theorem]{Remark}
\newtheorem{example}[theorem]{Example}
\theoremstyle{notations}

\theoremstyle{remarks}

\newcommand{\N}{\mathbb{N}}

\newcommand{\sub}{\subseteq}

\newcommand{\ov}{\overline}

\newcommand{\rg}{\rightarrow}
\newcommand{\lo}{\longrightarrow}

\newcommand{\wt}{\widetilde}
\newcommand{\vf}{\varphi}
\newcommand{\fr}{\frac}
\newcommand{\al}{\alpha}
\newcommand{\la}{\lambda}
\newcommand{\pa}{\partial}
\newcommand{\bt}{\beta}

\newcommand{\sq}{\simeq}

\newcommand{\psp}{\pi_1^{sp}(X,x)}

\newcommand{\pt}{\pi_1^{qtop}(X,x)}

\newcommand{\U}{\mathcal{U}}
\begin{document}

\author[H. Torabi, A. Pakdaman and B. Mashayekhy]
{Hamid Torabi$^\dag$, Ali Pakdaman$^\ddag$ and Behrooz Mashayekhy$^{\dag,*}$ }

\title[On topological fundamental groups of quotient spaces]
{On topological fundamental groups of quotient spaces}
\thanks{$^*$Corresponding author}
\thanks{E-mail addresses: hamid$_{-}$torabi86@yahoo.com; a.pakdaman@gu.ac.ir and bmashf@um.ac.ir}
\maketitle

\begin{center}
{\it $^\dag$Department of Pure Mathematics, Center of Excellence in Analysis on Algebraic Structures, Ferdowsi University of Mashhad,\\
P.O.Box 1159-91775, Mashhad, Iran.}\\
{\it $^\ddag$Department of Mathematics, Faculty of Science, Golestan University,\\
P.O.Box 155, Gorgan, Iran.}
\end{center}

\vspace{0.4cm}
\begin{abstract}
 Let $p:X\rightarrow X/A$ be a quotient map, where $A$ is a subspace of $X$. We explore conditions under which $p_*(\pi_1^{qtop}(X,x_0))$ is dense in $\pi_1^{qtop}(X/A,*))$, where the fundamental groups enjoy the natural quotient topology inherited from the loop space and $p_*$ is the induced continuous homomorphism by the quotient map $p$. Also, we give some applications to find out some properties for $\pi_1^{qtop}(X/A,*)$. In particular, we give some conditions in which $\pi_1^{qtop}(X/A,*)$ is an indiscrete topological group.
\end{abstract}
\vspace{0.5cm}
\section{Introduction and Motivation}
Let $p:(X,x_0)\rightarrow (Y,y_0)$ be a continuous map of pointed topological spaces. By applying the fundamental group functor on $p$ there exists the induced homomorphism $$p_*:\pi_1(X,x_0)\lo\pi_1(Y,y_0).$$
It seems interesting to relate the homology and homotopy groups of $X$ with that of $Y$ using properties of $p$. Vietoris  first studied the problem with his mapping theorem \cite{V}. Also, Smale first discovered an analog of Vietoris's mapping theorem hold for homotopy groups \cite{Sm}. Recently, Calcut, Gompf, and Mccarthy \cite{C2} proved a generalization of Smale's theorem as follows:

{\it Let $p:(X,x_0)\rightarrow (Y,y_0)$ be a quotient map of topological spaces, where $X$ is locally path connected and $Y$ is semilocally simply connected. If each fiber $p^{-1}(y)$ is connected, then the induced homomorphism $p_*:\pi_1(X,x_0)\rightarrow\pi_1(Y,y_0)$ is surjective}.

For a pointed topological space $(X,x_0)$ by $\pi_1^{qtop}(X,x_0)$ we mean the topological fundamental group endowed with the quotient topology inherited from the loop space under the natural map $\Omega(X,x_0)\lo\pi_1(X,x_0)$ that makes it a \emph{quasitopological group}. A \emph{quasitopological group} $G$ is a group with a topology such that inversion $g\lo g^{-1}$ and all translations are continuous. For more details, see \cite{B, Br, C1}.
It is known that this construction gives rise a homotopy invariant functor $\pi_1^{qtop}:hTop_*\lo qTopGrp$ from the homotopy category of based spaces to the category of quasitopological groups and continuous homomorphisms \cite{Br}.
Also, $\pi_1^{\tau}(X,x_0)$ is the fundamental group endowed with another topology introduced by Brazas \cite{Br2}. In fact, the functor $\pi_1^{\tau}$ removes the smallest number of open sets from the topology of $\pt$ so that makes it a topological group.

Let $X$ be a topological space and $A_1,A_2, . . . ,A_n$ be a finite collection of its subsets. The quotient space $X/(A_1, . . . ,A_n)$ is obtained from $X$ by identifying each of the sets $A_i$ to a point.
Now, let $(A,a)$ be a pointed subspace of $(X,a)$ and $p:(X,a)\lo (X/A,*)$ be the associated quotient map. In this paper, first we prove that if $A$ is an open subset of $X$ such that the closure of $A$, $\ov{A}$, is path connected, then the image of $p_*$ is dense in $\pi_1^{qtop}(X/A,*)$. Then by this fact, we show that the image of $p_*$ is dense in $\pi_1^{qtop}(X/(A_1,A_2,...,A_n),*)$, where the $A_i$'s are open subsets of $X$ with path connected closures and $p:X\lo X/(A_1,A_2,...,A_n)$ is the associated quotient map. Second, we prove that if $A$ is a closed subset of a locally path connected and first countable space $X$, then the image of $p_*$ is also dense in $\pi_1^{qtop}(X/A,*)$. By the two previous results we can show that the image of $p_*$ is dense in $\pi_1^{qtop}(X/(A_1,A_2,...,A_n),*)$, where $X$ is first countable, connected, locally path connected and the $A_i$'s are open or closed subsets of $X$ with disjoint path connected closures. Moreover, we give some conditions in which $p_*$ is an epimorphism. Also, by some examples, we show that $p_*$ is not necessarily onto. Finally, we give some applications of the above results to find out some properties of the topological fundamental group of the quotient space $X/(A_1,A_2,...,A_n)$. In particular, we prove that with the recent assumptions on $X$ and the $A_i$'s, $\pi_1^{qtop}(X/(A_1,A_2,...,A_n),*)$ is an indiscrete topological group when $X$ is simply connected. It should be mentioned that since the topology of $\pi_1^{\tau}(X,x_0)$ is coarser than $\pi_1^{qtop}(X,x_0)$, the above results can be obtained when we replace $\pi_1^{qtop}$ with $\pi_1^{\tau}$.
\section{Notations and preliminaries}
For a topological space $X$, by a path in $X$ we mean a continuous map $\al : [0, 1]\lo X$. The
points $\al(0)$ and $\al(1)$ are called the initial point and the terminal point of $\al$, respectively.
A loop $\al$ is a path with $\al(0)=\al(1)$. For
a path $\al:[0,1]\lo X$, $\al^{-1}$ denotes a path such that $\al^{-1}(t)=\al(1-t)$, for all $t\in [0,1]$.
Denote $[0,1]$ by $I$, two paths $\al, \beta:I\lo X$ with the same initial and terminal points are called homotopic relative to end points if there exists a continuous map $F:I\times I\lo X$ such that
\begin{displaymath}
F(t,s)= \left\{
\begin{array}{lr}
\al(t)    &       s=0 \\
\beta(t)    &      s=1\\
\al(0)=\beta(0)   &  t=0\\
\al(1)=\beta(1)    &  t=1.
\end{array}
\right.
\end{displaymath}
The homotopy is an equivalent relation and the homotopy class containing a path $\al$ is denoted by $[\al ]$. Since most of the homotopies that
appear in this paper have this property and end points are the same, we drop the term ``relative homotopy'' for
simplicity. For paths $\al, \beta:I\lo X$ with $\al(1)=\beta(0)$, $\al*\beta$ denotes the concatenation of $\al$ and $\beta$ that is a path
from $I$ to $X$ such that $(\al*\beta)(t)=\al(2t)$, for all $0\leq t\leq 1/2$ and $(\al*\beta)(t)=\beta(2t-1)$,
for all $1/2\leq t\leq1$.

For a pointed topological space $(X,x)$, let $\Omega(X,x)$ be the space of based maps from $I$ to $X$ with the compact-open topology. A subbase
for this topology consists of neighborhoods of the form $\langle K,U\rangle=\{\gamma\in\Omega(X,x)\ |\ \gamma(K)\sub U\}$, where $K\sub I$ is compact
and $U$ is open in $X$. When $X$ is path connected and the basepoint is clear, we just write $\Omega(X)$ and we will consistently denote the constant path at $x$ by $e_x$. The topological fundamental group of a pointed space $(X,x)$ can be described as the usual
fundamental group $\pi_1(X,x)$ with the quotient topology with respect to the canonical map $\Omega(X,x)\lo\pi_1(X,x)$ identifying
homotopy classes of loops, denoted by $\pi_1^{qtop}(X,x)$. A basic account of topological fundamental groups may be found in \cite{B}, \cite{C1} and \cite{Br}. For undefined notation, see \cite{M}.
\begin{definition}(\cite{A}).
A quasitopological group $G$ is a group with a topology such that inversion $G\lo G$, $g\mapsto
g^{-1}$, is continuous and multiplication $G \times G\lo G$ is continuous in each variable. A morphism of quasitopological groups is a continuous homomorphism.
\end{definition}
\begin{theorem} (\cite{Br}).
$\pi_1^{qtop}$ is a functor from the homotopy category of based topological spaces to the category of quasitopological groups.
\end{theorem}
A space $X$ is called {\it semi-locally simply connected} if for each point $x\in X$, there is an open
neighborhood $U$ of $x$ such that the inclusion $i : U\hookrightarrow X$ induces the trivial homomorphism $i_* : \pi_1(U,x)\lo \pi_1(X,x)$ or equivalently a loop in $U$ can be contracted inside $X$.
\begin{theorem} (\cite{Br}).
Let $X$ be a path connected space. If $\pi_1^{qtop}(X, x)$ is discrete for some $x\in X$, then $X$ is semi-locally simply connected. If $X$ is locally path connected and semi-locally simply connected, then $\pi_1^{qtop}(X, x)$ is discrete for all $x\in X$.
\end{theorem}
\section{Main Results}
In this section, $(A,a)$ is a pointed subspace of $(X,a)$, $p:(X,a)\lo (X/A,*)$ is the canonical quotient map so that $q:=p|_{X-A}:X-A\lo X/A-\{*\}$ is a homeomorphism. Also, by applying the functor $\pi_1^{qtop}$ on $p$ we have a continuous homomorphism $p_*:\pi_1^{qtop}(X,a)\lo\pi_1^{qtop}(X/A,*)$.
\begin{lemma}
If $A$ is an open subset of $X$, then any loops $\al:I\lo \ov{\{*\}}\sub X/A$ based at $*$ is nullhomotopic.
\end{lemma}
\begin{proof}
Define $F:I\times I\lo X/A$ by\\
\begin{displaymath}
F(t,s)= \left\{
\begin{array}{lr}
\al(t)    &       s=0 \\
*    &       s>0.
\end{array}
\right.
\end{displaymath}
If we prove that $F$ is continuous, then $F$ is a homotopy between $\al$ and $e_*$. For this, let $U$ be an open set in $X/A$. We show that $F^{-1}(U)$ is open in $I\times I$.\\
\emph{Case 1}: If $*\in U$, then $F^{-1}(U)=F^{-1}(\{*\})\cup (\al^{-1}(U)\times \{0\})$\\
 $$=((I\times (0,1])\cup (\al^{-1}(U)\times\{0\})$$
$$=(I\times (0,1])\cup (\al^{-1}(U)\times I)$$
which is open in $I\times I$.\\
\emph{Case 2}: If $*\notin U$, then $U\cap \partial \{*\}=\varnothing$ since if there exists $x\in \partial \{*\}$ such that $x\in U$, then $\{*\}\cap U\neq\varnothing$ which is a contradiction. Since $U\cap \ov{\{*\}}=(U\cap \{*\})\cup (U\cap\partial \{*\})=\varnothing$ and $\al(I)\subseteq \ov{\{*\}}$, we have $F^{-1}(U)=\varnothing$.
\end{proof}
\begin{theorem}
 Let $A$ be an open subset of $X$ such that $\overline{A}$ is path connected, then for each $a\in A$ the image of $p_*$ is dense in $\pi_1^{qtop}(X/A,*)$ i.e.
 $$\overline{p_*\pi_1^{qtop}(X,a)}=\pi_1^{qtop}(X/A,*) .$$
 \end{theorem}
\begin{proof}
\emph{Step One}: First, we show that for every loop $\al:(I,\partial I)\lo (X/A,*)$ such that $\al^{-1}(\{*\}^c)$ is connected, we have $[\al]\in Im(p_*)$.
 By assumption and openness of $\{*\}$ in $X/A$, there exist $s_1,s_2\in (0,1)$ such that $\al^{-1}(\{*\}^c)=[s_1,s_2]$. Since $\al^{-1}(\{*\}^c)$ is a compact subset of $I$, it suffices to let $s_1=inf\{\al^{-1}(\{*\}^c)\}$ and $s_2=sup\{\al^{-1}(\{*\}^{c})\}$. Let $\overline{\al}:[s_1,s_2]\lo X$ by $\overline{\al}(t)=q^{-1}(\al(t))$, then $\overline{\al}(s_1),\overline{\al}(s_2)\in \ov{A}$. Since if $G$ is an open neighborhood of $\overline{\al}(s_1)$ and $G\cap A=\varnothing$, then $q(G)=p(G)$ is an open neighborhood of $\al(s_1)$. Using continuity of $\al$, there exists an open neighborhood $J$ of $s_1$ in $I$ such that $\al(J)\subseteq G$. On the other hand, by definition of $s_1$, for all $s<s_1$, $\al(s)=*$ which implies that $*\in q(G)$ which is a contradiction since $*\notin Im(q)$. Similarly $\al(s_2)\in\ov A$. Since $\ov A$ is path connected, there exist two paths $\lambda_1:[0,s_1]\lo \ov A$ and $\lambda_2:[s_2,1]\lo \ov A$ such that $\lambda_1(0)=\lambda_2(1)=a,\ \lambda_1(s_1)=\overline{\al}(s_1)$ and $\lambda_2(0)=\overline{\al}(s_2)$. Define $\widetilde{\al}:I\lo X$ by
\begin{displaymath}
\wt{\al}(t)= \left\{
\begin{array}{lr}
\lambda_1(t)    &       0\leq t\leq s_1 \\
\ov{\al}(t)    &       s_1\leq t\leq s_2 \\
\lambda_2(t)    &       s_2\leq t\leq 1.
\end{array}
\right.
\end{displaymath}
By gluing lemma $\wt{\al}$ is continuous, so it remains to show that $p\circ \wt{\al}\simeq\al\ \ rel\{*\}$. Put $\al_1=\al|_{[0,s_1]}$, $\al_2=\al|_{[s_2,1]}$ and let $\vf_1:[0,1]\lo [0,s_1]$ and $\vf_2:[0,1]\lo [s_2,1]$ be linear homeomorphisms such that $\vf_1(0)=0$ and $\vf_2(0)=s_2$, then $p\circ \lambda_i\circ\vf_i\simeq\al_i\circ\vf_i,\ rel\{0,1\}$ since the $(p\circ \lambda_i\circ\vf_i)\circ(\al_i\circ\vf_i)^{-1}$'s are loops in $\overline{\{*\}}$ which by Lemma 3.1 are nullhomotopic.\\
\emph{Step Two}: By continuity of $\al$, $\al^{-1}(\{*\}^c)$ is a closed subset of $I$. Connected subsets of $I$ are intervals or one point sets, also connected components of $\al^{-1}(\{*\}^c)$ are closed in $\al^{-1}(\{*\}^c)$ and so they are compact in $I$. Therefore a component of $\al^{-1}(\{*\}^c)$ is either  closed interval or singleton. Given $[\al]\in \pi_1(X/A,*)$, we show that there exists a sequence of homotopy classes of loops $\{[\al_n]\}_{n\in N}$ in $Im(p_*)$ such that $[\al_n]\lo [\al]$ in $\pi_1^{qtop}(X/A,*)$.

We claim that the number of non-singleton components of $\al^{-1}(\{*\}^c)$ is countable. Let $S$ be the union of singleton components of $\al^{-1}(\{*\}^c)$ and for each $n\in\N$, $B_n$ be the set of non-singleton components of $\al^{-1}(\{*\}^c)$  with length at least $1/n$. Each $B_n$ is finite since if $B_n$ is infinite, then it has at least $n+1$ members. Therefore $\bigcup_{C\in B_n} C \subseteq \al^{-1}(\{*\}^c)\subseteq I$ which implies that $$(n+1)\times 1/n\leq\sum_{C\in B_n}diam(C)\leq diam(I)=1$$
 which is a contradiction. Thus each $B_n$ is finite which implies that $B=\bigcup_{n\in\N} B_n$ is countable. Rename elements of $B$ by $I_i=[a_i,b_i]$, $i\in J=\{1,2,...,s\} $, where $s=|B|$ if $B$ is finite and $i\in\mathbb{N}=J$ if $B$ is infinite. For every $n\in J$ define\\
\begin{displaymath}
{\al_n}(t)= \left\{
\begin{array}{lr}
\al(t)    &       t\in\bigcup_{i=1}^n[a_i,b_i] \\
*         &       otherwise. \\

\end{array}
\right.
\end{displaymath}
If $B$ is finite, put $\al_n=\al_s$, for every $n>s$.
We claim that the $\al_n$'s are continuous. For, if $V\subseteq X/A$ is open, then\\
i) If $*\in V$, then $$\al_1^{-1}(V)=[0,a_1)\cup (b_1,1]\cup \al|_{[a_1,b_1]}^{-1}(V)$$
which by continuity of $\al$ is open in $I$.\\
ii) If $*\notin V$, then we show that $\al_1^{-1}(V)=\al|_{[a_1,b_1]}^{-1}(V)=\al|_{(a_1,b_1)}^{-1}(V)$ which guaranties $\al^{-1}(V)$ is open. For this it suffices to show that $\al_1(a_1),\al_1(b_1)\notin V$. \\
For each $n\in \mathbb{N}$, $\al(a_n),\al(b_n)\in \ov{\{*\}}$ and $\{\al(a)|\ a\in S\}\subseteq \ov{\{*\}}$. For, if $G$ is an open neighborhood of $\al(a_n)$, then $ \al^{-1}(G)$ is an open neighborhood of $a_n$, so there exists $\varepsilon>0$ such that $(a_n-\varepsilon,a_n+\varepsilon)\subseteq\al^{-1}(G)$ or equivalently $\al((a_n-\varepsilon,a_n+\varepsilon))\subseteq G$. If $*\notin G$, then $(a_n-\varepsilon,a_n+\varepsilon)\subseteq \al^{-1}(\{*\}^c)$ which is a contradiction since $[a_n,b_n]$ is a connected component of $\al^{-1}(\{*\}^c)$. Similarly $\al(b_n)\in\ov{\{*\}}$ for each $n\in\N$ and $\al(S)\sub\ov{\{*\}}$.
Thus if $\al_1(b_1)=\al(b_1)\in V$, then $V$ must meet $\{*\}$ which is a contradiction since $*\notin V$. Therefore $\al_1$ is a continuous loop such that $[\al_1]\in Im(p_*)$. Similarly, all the $\al_n$'s are continuous. Also, for every $n\in\N$, $[\al_n]$ is a product of n homotopy classes of loops which are similar to loops introduced in Step 1. This implies that $[\al_n]\in Im(p_*)$ since $Im(p_*)$ is a subgroup.

Now we show that the sequence $\{\al_n\}$ converges to $\al$. Let $\al\in \langle K,U\rangle$, where $K$ is a compact subset of $I$ and $U$ is an open subset of $X/A$, then\\
i) If $*\in U$, then for each $n\in\mathbb{N}$, $\al_n\in <K,U>$ since for each $t\in K$, $\al_n(t)=\al(t)$ or $\al_n(t)=*$ which in both cases $\al(t)\in U$.\\
ii) If $K\subseteq\al^{-1}(\{*\}^c)$ and $K\cap S\neq\varnothing$, then there exists $a\in K\cap S\subseteq\al^{-1}(U)$, so $\al(a)\in U$, but $\al(a)\in \ov{\{*\}}$ and $U$ is open. Thus $*\in U$ and by (i), for each n, $\al_n(K)\subseteq U$.\\
iii) If $K\subseteq\bigcup_{n=1}^\infty I_n$ and there exist $n_1,n_2,...,n_s$ such that $K\subseteq\bigcup_{i=1}^sI_{n_i}$, then by definition of the $\al_n$'s, for each $n\geq max\{n_1,...,n_s\}$ we have $\al_n(K)\subseteq U$.\\
vi) If $K\subseteq\bigcup_{n=1}^\infty I_n$ and there exists an infinite subsequence $\{I_{n_r}\}$ such that $K\cap I_{n_r}\neq\varnothing$, then there is a sequence $\{x_{n_r}|x_{n_r}\in K\cap I_{n_r}\}$ such that it has a subsequence $\{x_{n_{r_s}}\}$ converges to an element of $K$, $b$ say, by compactness of $K$. Since $b\in K\subseteq\al^{-1}(U)$, there exists $\varepsilon>0$ such that $(b-\varepsilon,b+\varepsilon)\subseteq\al^{-1}(U)$. Also, there exists $s_0$ such that for each $s\geq s_0$ , $x_{n_{r_s}}\in (b-\varepsilon /2,b+\varepsilon /2)$ since $x_{n_{r_s}}\lo b$. Since $diam(I_n)\lo 0$, there is $n_0$ such that for each $n\geq n_0$ , $diam(I_n)<\varepsilon /2$. Choose $s_1\in N$ such that $s_1\geq s_0$ and $n_{r_{s_1}}\geq n_0$, then $x_{n_{r_{s_1}}}\in (b-\varepsilon /2,b+\varepsilon /2)$. Also $x_{n_{r_{s_1}}}\in K\cap I_{n_{r_{s_1}}}$ and $diam(I_{n_{r_{s_1}}})<\varepsilon /2$, so $a_{n_{r_{s_1}}}\in I_{n_{r_{s_1}}}\subseteq (b-\varepsilon ,b+\varepsilon )\subseteq \al^{-1}(U)$ which implies that $\al(a_{n_{r_{s_1}}})\in U$ and therefore $*\in U$ since $\al(a_{n_{r_{s_1}}})\in\pa(\{*\})$. Using the last (i) we have $\al_n(K)\subseteq U$, for each $n\in\N$.\\
\end{proof}
\begin{definition}Let $X$ be a topological space and $A_1,A_2, . . . ,A_n$ be any subsets of $X$, $n\in\N$. By the quotient space $X/(A_1, . . . ,A_n)$ we mean the quotient space obtained from $X$ by identifying each of the sets $A_i$ to a point. Also, we denote the associated quotient map by $p:X\lo X/(A_1,A_2,...,A_n)$.
\end{definition}
\begin{corollary}
If $A_1,A_2$ are open subsets of a path connected space $X$ such that $\ov{A_1},\ov{A_2}$ are path connected. Then for every $a\in A_1\cup A_2$ the following equality holds: $$\ov{p_*\pi_1^{qtop}(X,a)}=\pi_1^{qtop}(X/(A_1,A_2),*).$$
\end{corollary}
\begin{proof}
We can assume that the $A_i$'s are disjoint. If they are not disjoint, the result follows from Theorem 3.2. Let $p_1:X\lo X/A_1$, $p_2:X/A_1\lo X/(A_1,A_2)$ be associated quotient maps and $a_1=a\in A_1$. By Theorem 3.2, $\ov{(p_1)_*\pi_1^{qtop}(X,a_1)}=\pi_1^{qtop}(X/A_1,*_1)$, where $*_1=p_1(a_1)$. Since $X$ is path connected, so is $X/A_1$. Also, $p_1(A_2)$ is an open subset of $X/A_1$ and the closure of $p_1(A_2)$ in $X/A_1$ is path connected. Let $a_2\in p_1(A_2)$, then $\ov{(p_2)_*\pi_1^{qtop}(X/A_1,a_2)}=\pi_1^{qtop}(X/(A_1,A_2),*_2)$, where $*_2=p_2(a_2)$. Since $X/A_1$ is path connected, there exists a homeomorphism $\vf_1:\pi_1^{qtop}(X/A_1,*_1)\lo\pi_1^{qtop}(X/A_1,a_2)$ by $\vf_1([\al])=[\gamma*\al*\gamma^{-1}]$, where $\gamma$ is a path from $a_2$ to $*_1$. We have $\ov{(p_2)_*\circ\vf_1\circ (p_1)_*(\pi_1^{qtop}(X,a))}\supseteq ((p_2)_*\circ\vf_1)(\ov{(p_1)_*(\pi_1^{qtop}(X,a))})=((p_2)_*\circ\vf_1)(\pi_1^{qtop}(X/A_1,*_1))=(p_2)_*\pi_1^{qtop}(X/A_1,a_2)=Im(p_2)_*$ which implies that $Im(p_2)_*\circ\vf_1\circ(p_1)_*$ is dense in $\pi_1(X/(A_1,A_2),*_2)$. If $\gamma'=p_2\circ\gamma$, then $\vf_2:\pi_1^{qtop}(X/(A_1,A_2),*_2)\lo\pi_1^{qtop}(X/(A_1,A_2),*_1)$ by $\vf_2([\al])=[\gamma'^{-1}*\al*\gamma']$ is a homeomorphism. Hence $Im(\vf_2\circ(p_2)_*\circ\vf_1\circ (p_1)_*)$ is dense in $\pi_1^{qtop}(X/(A_1,A_2),*_1)$. Moreover $\vf_2\circ(p_2)_*\circ\vf_1\circ (p_1)_*=p_*$ which implies that $Im(p_*)$ is dense in $\pi_1^{qtop}(X/(A_1,A_2),*)$, as desired.
\end{proof}

By induction and Corollary 3.4, we have the following results.
\begin{corollary}
Let $A_1,A_2,...,A_n$ be open subsets of a path connected space $X$ such that the $\ov{A_i}$'s are path connected for each $i=1,2,...,n$. Then for any $a\in\bigcup_{i=1}^nA_i$ the following equality holds:
 $$\ov{p_*\pi_1^{qtop}(X,a)}=\pi_1^{qtop}(X/(A_1,A_2,...,A_n),*).$$
\end{corollary}
\begin{corollary}
Let $A_1,A_2,...,A_n$ be open subsets of a connected, locally path connected space $X$ such that the $\ov{A_i}$'s are path connected for every $i=1,2,...,n$. If $X/(A_1,A_2,...,A_n)$ is semi-locally simply connected, then for each $a\in\bigcup_{i=1}^nA_i$, $p_*:\pi_1^{qtop}(X,a)\lo\pi_1^{qtop}(X/(A_1,A_2,...,A_n,*)$ is an epimorphism.
\end{corollary}
\begin{proof}
Let $U$ be an open neighborhood of $\bar x\in (X/A_1)\setminus\ov{\{*\}}$. Since $X$ is locally path connected, there is a path connected open neighborhood $\wt{U}\sub p^{-1}(U)$ of $x=q^{-1}(\bar x)$ such that $\wt{U}\cap A_1=\varnothing$. Then $V:=p(\wt{U})=q(\wt{U})\sub U$ is a path connected open neighborhood of $\bar x$. $X/A_1$ is locally path connected at $*$ since $\{*\}$ is an open subset of $X/A_1$. Let $U$ be an open neighborhood of $\bar x\in \pa(\{*\})$, then there exists a path connected open neighborhood $\wt{U}\sub p^{-1}(U)$ of $x$. Since $p^{-1}(p(\wt {U}))=\wt{U}\cup A_1$, $p(\wt{U})$ is a path connected open neighborhood of $\bar x$ in $U$. Therefore $X/A_1$ is locally path connected. Similarly $X/(A_1,A_2,...,A_n)$ is connected, locally path connected. Since $X/(A_1,A_2,...,A_n)$ is a connected, semi-locally simply connected and locally path connected space, by Theorem 2.3, $\pi_1^{qtop}(X/(A_1,A_2,...,A_n),*)$ is a discrete topological group which implies that $Im(p_*)=\pi_1^{qtop}(X/(A_1,A_2,...,A_n),*)$ by Corollary 3.5.
\end{proof}
In the following example we show that with the assumptions of Theorem 3.2, $p_*$ is not necessarily onto.
\begin{example}
Let $A_n=\{1/(2n-1),1/2n\}\times [0,1+1/2n]\bigcup [1/2n,1/2n-1]\times\{1+1/2n\}$ for each $n\in\N$. Consider $X=(\bigcup_{n\in\N}A_n)\bigcup\{0\}\times [0,1]\bigcup[0,1]\times\{0\}$ with $a=(0,0)$ as the base point and $A=\{(x,y)\in X\ |\ y<1\}$ (see Figure 1). $A$ is an open subset of $X$ with path connected closure. Assume $I_n=(1/2+1/2(n+1),1/2+1/2n]$ and $f_n$ be a homeomorphism from $I_n$ to $A_n-\{(1/2n,0)\}$ for every $n\in\mathbb{N}$. Define $f:I\lo X$ by
 \begin{displaymath}
{f}(t)= \left\{
\begin{array}{lr}
the\ point\ (0,2t)    &       t\in [0,1/2], \\
f_n(t)       &       t\in I_n. \\

\end{array}
\right.
\end{displaymath}
We claim that $\al=p\circ f$ is a loop in $X/A$ at $*$. It suffices to show that $\al$ is continuous on $t=1/2$ and boundary points of $I_n$'s since $f$ is continuous on $[0,1/2)$ and by gluing lemma on $\bigcup int(I_n)$. Since $\al$ is locally constant at $t=1/2+1/2n$ for each $n\in\mathbb{N}$, $\al$ is continuous at boundary points of $I_n$. For each open neighborhood $G$ of $f(1/2)=(0,1)$ in $X$, there exists $n_0\in\mathbb{N}$ such that $G$ contains $A_n\bigcap A^c$ for $n>n_0$. Therefore continuity at $t=1/2$ follows from $\al(1/2)\in \overline{\{*\}}$. Now let $B\sub\mathbb{N}$ and define
\begin{displaymath}
{g_B}(t)= \left\{
\begin{array}{lr}
(p\circ f)(t)    &       t\in\bigcup_{m\in B}I_m, \\
*       &       otherwise. \\

\end{array}
\right.
\end{displaymath}
Then $g_B$ is continuous and for $B_1,B_2\sub\mathbb{N}$ such that $B_1\neq B_2$, $[g_{B_1}]\neq[g_{B_2}]$ which implies that $\pi_1(X/A,*)$ is uncountable. But by compactness of $I$, a given path in $X$ can traverse finitely many of the $A_n$'s and therefore $\pi_1(X,a)$ is a free group on countably many generators which implies that $p$ does not induce a surjection of fundamental groups.
\end{example}
\begin{figure}
 \includegraphics[scale=0.4]{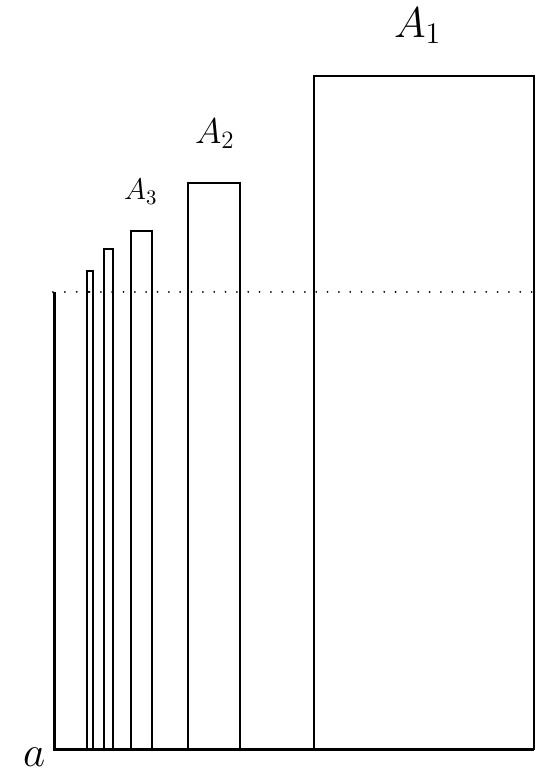}
  \caption{}\label{1}
\end{figure}

Let $(X,x)$ be a pointed topological space such that $\{x\}$ is closed. If $\al: [0, 1]\lo X$ is a loop in $X$ based at $x$,
 then $\al^{-1}(\{x\})$ is a closed subset of $[0,1]$. Its complement, $\al^{-1}(\{x\}^c)$ is therefore the union of a countable collection of
disjoint open intervals. We denote this collection of intervals by $W_{\al}$.
\begin{definition}
Let $(X,x)$ be a pointed topological space. A loop $\al$ in $X$ based at $x$ is called {\bf semi-simple} if $W_{\al}=\{(0,1)\}$ and is called {\bf geometrically simple} if $W_{\al}$ has one element. If $W_{\al}$ is finite, then the loop $\al$ is called {\bf geometrically finite} \cite{MM}.
\end{definition}
\begin{lemma}
Every geometrically simple loop is homotopic to a semi-simple loop.
\end{lemma}
\begin{proof}
Let $\al$ be a geometrically simple loop at $x\in X$. Then there are $r,s\in [0,1]$ such that $\al^{-1}(\{x\}^c)=(r,s)$ and $\al(r)=\al(s)=x$. Let $\bt:=\al|_{[r,s]}$ and $\vf:[0,1]\lo [r,s]$ be a linear homeomorphism, then $\bt\circ\vf$ is a semi-simple loop at $x$ and $\al\sq\bt\circ\vf$.
\end{proof}
In the sequel, for a semi-simple loop $\al:I\lo X/A$ denote $\wt\al=q^{-1}\circ{\al|_{(0,1)}}:(0,1)\lo (X-A)$, where $A$ is a closed subset of a topological space $X$.
\begin{lemma}
 Let $A\subseteq X$ be a closed subset of $X$ and $\al$ be a semi-simple loop at $*$ in $X/A$. If $lim_{t\rightarrow 0}\wt\al(t)$ and $lim_{t\rightarrow 1}\wt\al(t)$ do not exist, then for each $t_0\in (0,1)$, there are  $b_0,b_1\in A$ such that $b_0$ is a limit point of $\wt\al((0,t_0))$ and $b_1$ is a limit point of $\wt\al((t_0,1))$.
\end{lemma}
\begin{proof}
Let $t_0\in (0,1)$ and by contrary suppose that each $b\in A$ has an open neighborhood $G_b$ such that $G_b\cap \wt\al((0,t_0))=\varnothing$. Then $G=\bigcup_{b\in A} G_b$ is an open neighborhood of $A$ and so $p(G)$ is an open neighborhood of $*$ (since $p^{-1}(p(G))=G$) such that does not intersect $\al((0,t_0))$ which is a contradiction to continuity of $\al$.
\end{proof}
\begin{theorem}
If $A$ is a closed path connected subset of a locally path connected space $X$ such that every point of $A$ has a countable local base in $X$, then for each $a\in A$ we have $$\overline{p_*\pi_1^{qtop}(X,a)}=\pi_1^{qtop}(X/A,*).$$
\end{theorem}
\begin{proof}
\emph{Step One}: Let $[\al]\in \pi_1(X/A,*)$, where $\al$ is a semi-simple loop in $X/A$ at $*$. \\
\emph{Case 1}: Assume $a_0=lim_{t\rightarrow 0}\wt\al(t)$ and $a_1=lim_{t\rightarrow 1}\wt\al(t)$ exist, so $a_0$,$a_1\in A$ and we can define a path $\ov\al:I\lo X$ such that $\ov\al|_{(0,1)}=\wt\al$, $\ov\al(0)=a_0$, $\ov\al(1)=a_1$. Since $A$ is path connected, there exist paths $\la_0,\la_1:I\lo A$ such that $\lambda_0$ is a path from $a$ to $a_0$ and $\lambda_1$ is a path from $a_1$ to $a$. Therefore $\la_0*\ov\al*\la_1$
is a loop at $a$ such that $p_*([\la_0*\ov\al*\la_1])=[\al]$.\\
\emph{Case 2}: If at least one of the above limits does not exist, then we make a sequence $\{[\al_n]\}_{n\in\N}$ in $Im(p_*)$ so that converges to $[\al]$.
Without lost of generality, we can assume that $a_0=lim_{t\rg 0}\wt\al(t)$ exists and $a_1\in A$ is a limit point of $\wt\al((1/2,1))$ by Lemma 3.10. We can define a continuous map $\ov\al:[0,1)\lo X$ such that $\ov\al|_{(0,1)}=\wt\al$, $\ov\al(0)=a_0$. By hypothesis, there is a countable local base $\{O_i\}_{i\in\N}$ at $a_1$. Let $\{G_i\}_{i\in\N}$ be a sequence of open neighborhoods of $a_1$ such that $G_i=O_1\cap ...\cap O_i$. Since $X$ is locally path connected and the $G_i$'s are open neighborhoods of $a_1$, there exist path connected open neighborhoods  $G_i'\sub G_i$ of $a_1$. Since the point $a_1$ is a limit point, there are $t_i\in (1/2,1)$ such that $\wt\al(t_i)\in G_i'$, $t_i<t_{i+1}$, $t_n\lo 1$ and there are paths $\gamma_i:[t_i,1]\lo G_i'$ from $\ov\al(t_i)$ to $a_1$, for all $i\in\mathbb{N}$. Since $A$ is path connected, there exist paths $\la_0,\la_1:I\lo A$ such that $\lambda_0$ is a path from $a$ to $a_0$ and $\lambda_1$ is a path from $a_1$ to $a$. Let $\ov\al_n:=\la_0*\ov\al|_{[0,t_n]}\circ\xi_n*\gamma_n\circ\zeta_n*\la_1$, where $\xi_n:[0,1]\lo [0,t_n]$ and $\zeta_n:[0,1]\lo [t_n,1]$ are increasing linear homeomorphisms. Note that every $\ov\al_n$ is a loop in $X$ at $a$ and if $\beta_n:=p\circ(\gamma_n\circ\zeta_n)$, then $$\al_n':=p\circ\ov\al_n=e_**\al|_{[0,t_n]}\circ\xi_n*\beta_n*e_*$$ is a loop in $X$ at $*$ and $p_*([\ov\al_n])=[\al_n']$. Define $\al_n:I\lo X/A$ by
 \begin{displaymath}
{\al_n}(t)= \left\{
\begin{array}{lr}
\al(t)    &       0\leq t\leq t_n \\
p\circ\gamma_n(t)       &       t_n\leq t\leq 1 \\

\end{array}
\right.
\end{displaymath}
which is a loop at $*$ and $[\al_n']=[\al_n]$. Thus it suffices to prove that $\al_n\lo\al$. If $\al\in\langle K,U\rangle$, where $K$ is a compact subset of $[0,1]$ and $U$ is an open subset of $X/A$, then\\
i) If $*\notin U$, then $K\cap \al^{-1}(*)=\varnothing$. Let $m\in\mathbb{N}$ such that $t_{m}\geq max\ K$. Since for each $t\in K$, $t\leq t_{m}$, we have $\al_n(t)=\al(t)\in U$, for all $n>m$ which implies that $\al_n(K)=\al(K)\sub U$, for each $n> m$.\\
ii) If $*\in U$, then $p^{-1}(U)$ is an open neighborhood of $A$, thus there exists $m\in\mathbb{N}$ such that $G_n'\sub p^{-1}(U)$, for each $n\geq m$. Therefore $Im (\gamma_n)\sub G_n'$ which implies that $Im( p\circ\la_n)\sub U$. Thus for all $t\in K$ and $n\geq m$ we have\\
\begin{displaymath}
{\al_n}(t)= \left\{
\begin{array}{lr}
\al(t)\in U    &       t\in [0,t_n] \\
(p\circ\gamma_n)(t)\in U         &       t\in [t_n,1]. \\

\end{array}
\right.
\end{displaymath}
Therefore for a semi-simple loop $\al$ in $X/A$ we have $[\al]\in \ov{Im(p_*)}$. Similarly, for every loop $\al$ such that $\al^{-1}(\{*\})$ is finite, $[\al]\in\ov{Im(p_*)}$ which implies that the homotopy class of every geometrically finite loop belongs to $\ov{Im(p_*)}$ by Lemma 3.9. \\
\emph{Step Two}:
If $\al$ is not geometrically finite, $W_{\al}$ is countable since every open subset of $I$ is a countable union of open intervals. Let $I_j=\ov {L_j}$ where $W_{\al}=\{L_j|j\in\N\}$ and let\\
\begin{displaymath}
{\al_j}(t)= \left\{
\begin{array}{lr}
\al(t)   &       t\in I_1\cup ...\cup I_j \\
*        &       otherwise, \\

\end{array}
\right.
\end{displaymath}
then $[\al_j]\in \ov {Im(p_*)}$ since the $\al_j$'s are geometrically finite. \\
 Since $\ov{(\ov{ Im(p_*)})}=\ov{ Im(p_*)}$, it suffices to show that $\al_j\lo \al$. For, if $\al\in\langle K,U\rangle$ for a compact subset $K$ of $[0,1]$ and an open subset $U$ of $X/A$, then\\
i) If $*\in U$, then for each $t\in K$ and $j\in\N$, $\al_j(t)$ takes value $\al(t)$ or $*$ which in both cases belongs to $U$, so $\al_j(K)\sub U$, for all $j\in\N$.\\
ii) If $*\notin U$, then $K\cap\al^{-1}(\{*\})=\varnothing$, so $K\sub\cup_j L_j$. By compactness of $K$ we have $K\sub\cup_sL_{j_s}$, for $s=1,2,...,n_K$. Let $M=max\{j_s|s=1,2,...,n_K\}$, then $\al_j(K)=\al(K)\sub U$, for each $j\geq M$.\\
\end{proof}
\begin{corollary}
Let $A_1,A_2,...,A_n$ be disjoint path connected, closed subsets of a first countable, connected, locally path connected space $X$. Then for every $a\in\bigcup_{i=1}^nA_i$ the following equality holds:
 $$\ov{p_*\pi_1^{qtop}(X,a)}=\pi_1^{qtop}(X/(A_1,A_2,...,A_n),*).$$
\end{corollary}
\begin{proof}
Let $p_i:X/(A_1,A_2,...,A_{i-1})\lo X/(A_1,A_2,...,A_i)$. Since every point of each $p_{i-1}(A_i)$ has a countable local base in the connected, locally path connected space $X/(A_1,A_2,...,A_{i-1})$, by Theorem 3.11 the result holds.
\end{proof}
\begin{corollary}
Let $A_1,A_2,...,A_n$ be disjoint path connected, closed subsets of a first countable, connected, locally path connected space $X$ such that $X/(A_1,A_2,...,A_n)$ is semi-locally simply connected. Then for each $a\in\bigcup_{i=1}^nA_i$, $p_*:\pi_1(X,a)\lo\pi_1(X/(A_1,A_2,...,A_n,*)$ is an epimorphism.
\end{corollary}
\begin{proof}
Since $X/(A_1,A_2,...,A_n)$ is connected, locally path connected and semi-locally simply connected space, $\pi_1^{qtop}(X/(A_1,A_2,...,A_n),*)$ is a discrete topological group which implies that $Im(p_*)=\pi_1(X/(A_1,A_2,...,A_n),*)$ by Corollary 3.12.
\end{proof}
In the following example, we show that the condition ``path connectedness for $A$" is necessary in Theorem 3.11.
\begin{example}
Let $A=\{(1,0),(0,1)\}\subset X=S^1$. Clearly $X/A$ is homeomorphic to the Figure 8 space, $S^1\bigvee S^1$. Since $X$ and $X/A$ are locally path connected and semi-locally simply connected
$$p_*:\pi_1^{qtop}(X,0)\cong\mathbb{Z}\lo\pi_1^{qtop}(X/A,*)\cong\mathbb{Z}*\mathbb{Z}$$
is a continuous homomorphism of discrete topological spaces. Since the free product $\mathbb{Z}*\mathbb{Z}$ is not abelian, $p_*$ is not onto and since $\pi_1^{qtop}(X/A,*)$ is discrete,
$Im(p_*)$ is not dense in $\pi_1^{qtop}(X/A,*)$.
\end{example}
In the following example, we show that the condition ``locally path connectedness for $X$" is necessary in Theorem 3.11.
\begin{example}
Let $X_1=\{(x,sin(2\pi/x))\in \mathbb{R}^2|\ 0<x\leq1\}$, $X_2=\{(x,y)\in \mathbb{R}^2|\ x^2+\frac{y^2}{4}=1\ ,\ y\leq0\}$, $X_3=\{(x,0)\in \mathbb{R}^2|\ -1\leq x\leq0\}$ and $A=\{(0,y)\in \mathbb{R}^2|\ -1\leq y\leq1\}$. If $X=X_1\cup X_2\cup X_3\cup A$, then $\pi_1(X,x_0)=0$ and $\pi_1(X/A,*)\cong Z$. Since $X/A$ is a locally path connected and semi-locally simply connected space, $\pi_1^{qtop}(X/A,*)$ is discrete which implies that $\ov{p_*(\pi_1^{qtop}(X,x_0))}\neq \pi_1^{qtop}(X/A,*)$.
\end{example}
In the next example, we show that with the assumptions of Theorem 3.11, $p_*$ is not necessarily an epimorphism and hence the hypothesis semi-locally simply connectedness in Corollary 3.13 is essential.
\begin{example}
Let $C_n=\{(x,y)\in \mathbb{R}^2|\ (x-\fr{1}{n})^2+y^2=\fr{1}{n^2}\}$, for $n\in\mathbb{N}$, $HE_o=\bigcup_{n\in\mathbb{N}}C_{2n-1}$, $HE_e=\bigcup_{n\in\mathbb{N}}C_{2n}$ and $X=\left(HE_0\times\{0\}\right)\cup\left(HE_e\times\{1\}\right)\cup A$, where $A=\left(\{(0,0)\}\times I\right)$. One can easily see that $X/A$ is the Hawaiian Earing space. Let $\al$ be the loop in $X/A$ that traverse $p(C_1), p(C_2),...$ in ascending order. By the structure of the fundamental group of the Hawaiian Earing \cite{con} we have $[\al]\notin Im(p_*)$ since if $p_*([\bt])=[\al]$, then the loop $\bt$ must traverse infinitely many times $A$ which is a contradiction to the continuity of $\bt$.
\end{example}
\begin{corollary}
Let $A_1,A_2,...,A_n$ be subsets of a first countable, connected, locally path connected space $X$ with disjoint path connected closure such that each $A_i$ is closed or open. Then for any $a\in\bigcup_{i=1}^nA_i$ we have
 $$\ov{p_*\pi_1^{qtop}(X,a)}=\pi_1^{qtop}(X/(A_1,A_2,...,A_n),*).$$
\end{corollary}
\begin{proof}
By changing the order, we can assume that $A_1,...,A_k$ are closed and $A_{k+1},...A_n $ are open, for a $1\leq k\leq n$. By applying Corollary 3.12 we have
$$\ov{q_*\pi_1^{qtop}(X,a)}=\pi_1^{qtop}(X/(A_1,A_2,...,A_k),*),$$
where $q:X\rightarrow X/(A_1,A_2,...,A_k)$ is the natural quotient map.
Consider the natural quotient map
 $r:X/(A_1,A_2,...,A_k)\rightarrow X/(A_1,...,A_k,...,A_n)$. Note that $p=r\circ q$ and 
 since the $A_j$'s have disjoint path connected closures, $\overline{A_j}$ is also path connected in $X/(A_1,...,A_k)$, for all $j>k$.
 Now, using Corollary 3.5 the result holds.
\end{proof}
\begin{remark}
Note that since the topology of $\pi_1^{\tau}(X,x)$ is coarser than $\pi_1^{qtop}(X,x)$, the results of this section can be restated for $\pi_1^{\tau}$ when we replace $\pi_1^{qtop}$ with $\pi_1^{\tau}$.
\end{remark}

\section{some applications}
It seems interesting to investigate on the topology of quasitopological fundamental groups and some people have found some properties of this topology (see \cite{B,Br,Br2,C1,P4,P2,P1,T2}).
In this section, we intend to give some applications of the results of the previous section to find out some properties of the topological fundamental group of the quotient space $X/(A_1,A_2,...,A_n)$.
 By $(X,A_1,A_2,...,A_n)$ we mean an $(n+1)$-$tuple$ of spaces with one of the following conditions $(\clubsuit)$:\\
 (i): The $A_i$'s are open subsets of $X$ with path connected closures.\\
 (ii): $X$ is a connected, locally path connected, first countable space and the $A_i$'s are closed subsets of $X$ with disjoint path connected closures.

\begin{theorem}
For an $(n+1)$-$tuple$ of spaces $(X,A_1,A_2,...,A_n)$ with the assumption $(\clubsuit)$, if $X$ is simply connected, then $\pi_1^{qtop}(X/(A_1,A_2,...,A_n),*)$ is an indiscrete topological group.
\end{theorem}
\begin{proof}
Since $X$ is simply connected, $p_*\pi_1^{qtop}(X,a)=\{[e_*]\}$, where $e_*$ is the constant loop at $*$ in $X/(A_1,A_2,...,A_n)$. Then by Corollaries 3.5 and 3.17 $\{[e_*]\}$is a dense subset of $\pi_1^{qtop}(X/(A_1,A_2,...,A_n),*)$. Since $\pi_1^{qtop}(X/(A_1,A_2,...,A_n),*)$ is a quasitopological group, for every $[\al]\in \pi_1^{qtop}(X/(A_1,A_2,...,A_n),*)$, the left multiplication $L_{[\al]}:\pi_1^{qtop}(X/(A_1,A_2,...,A_n),*)\lo \pi_1^{qtop}(X/(A_1,A_2,...,A_n),*)$ given by $L_{[\al]}([\bt])=[\al*\bt]$ is a homeomorphism which implies that $\{[\al]\}$ is also dense
in $\pi_1^{qtop}(X/(A_1,A_2,...,A_n),*)$. Hence every nonempty open subset of $\pi_1^{qtop}(X/(A_1,A_2,...,A_n),*)$ contains every element $[\al]$ of $\pi_1^{qtop}(X/(A_1,A_2,...,A_n),*)$
which implies that $\pi_1^{qtop}(X/(A_1,A_2,...,A_n),*)$ is an indiscrete topological group.
\end{proof}

\begin{theorem}
For an $(n+1)$-$tuple$ of spaces $(X,A_1,A_2,...,A_n)$ with the assumption $(\clubsuit)$, if $\pi_1^{qtop}(X,a)$ is compact and $\pi_1^{qtop}(X/(A_1,A_2,...,$ $A_n),*)$ is Hausdorff, then the quasitopological fundamental group $\pi_1^{qtop}(X/(A_1,A_2,...,A_n),*)$ is either a discrete topological group or uncountable.
\end{theorem}
\begin{proof}
If $\pi_1^{qtop}(X/(A_1,A_2,...,A_n),*)$ has at least one isolated point, then every singleton is open since left translations
$$L_{[\al]}:\pi_1^{qtop}(X/(A_1,A_2,...,A_n),*)\lo \pi_1^{qtop}(X/(A_1,A_2,...,A_n),*)$$
are homeomorphisms, for every $[\al]\in\pi_1^{qtop}(X/(A_1,A_2,...,$ $A_n),*)$. Thus \\ $\pi_1^{qtop}(X/(A_1,A_2,...,A_n),*)$ is a discrete topological group. It is a well-known result that a nonempty compact Hausdorff space without isolated points is uncountable \cite[Theorem 27.7]{M}. Hence if $\pi_1^{qtop}(X/(A_1,A_2,...,$ $A_n),*)$ has no isolated points, then in order to show that $\pi_1^{qtop}(X/(A_1,A_2,...,$ $A_n),*)$ is uncountable it is enough to show that $\pi_1^{qtop}(X/(A_1,A_2,...,$ $A_n),*)$ is compact.
By Corollaries 3.5 and 3.17,  
$$\ov{p_*\pi_1^{qtop}(X,a)}=\pi_1^{qtop}(X/(A_1,A_2,...,A_n),*).$$ 
Since $\pi_1^{qtop}(X,a)$ is compact and $p_*$ is continuous  $p_*\pi_1^{qtop}(X,a)$ is compact in $\pi_1^{qtop}(X/(A_1,A_2,...,A_n),*)$. Since $\pi_1^{qtop}(X/(A_1,A_2,...,A_n),*)$ is Hausdorff, $p_*\pi_1^{qtop}(X,a)$ is closed in $\pi_1^{qtop}(X/(A_1,A_2,...,A_n),*)$ and so
$p_*\pi_1^{qtop}(X,a)=\pi_1^{qtop}(X/(A_1,A_2,...,A_n),*)$. Hence $\pi_1^{qtop}(X/(A_1,A_2,...,A_n),*)$ is compact and so it is uncountable.
\end{proof}
\begin{corollary}
For an $(n+1)$-$tuple$ of spaces $(X,A_1,A_2,...,A_n)$ with the assumption $(\clubsuit)$, if $\pi_1^{qtop}(X,a)$ is a compact, countable quasitopological group, then either $X/(A_1,A_2,...,A_n)$ is semi-locally simply connected or $\pi_1^{qtop}(X/(A_1,A_2,...,A_n),*)$ is not Hausdorff.
\end{corollary}
\begin{proof}
Let $\pi_1^{qtop}(X/(A_1,A_2,...,A_n),*)$ be Hausdorff, then by a similar proof of Theorem 4.2 $p_*$ is onto. Therefore $\pi_1^{qtop}(X/(A_1,A_2,...,A_n),*)$ is countable since $\pi_1^{qtop}(X,a)$ is countable. Theorem 4.2 implies that $\pi_1^{qtop}(X/(A_1,A_2,...,A_n),*)$ is a discrete topological groups. Hence by Theorem 2.3 $X/(A_1,A_2,...,A_n)$ is semi-locally simply connected.
\end{proof}

If $\U$ is an open cover of a connected and locally path connected space $X$, then the subgroup of $\pi_1(X, x)$ consisting of all homotopy classes of loops that can be represented by a
product of the following type: $$\prod\limits_{j=1}^{n}u_jv_ju_j^{-1},$$
where the $u_j$'s are arbitrary paths (starting at the base point $x$) and each $v_j$ is a loop inside one of the neighborhoods $U_i\in\mathcal{U}$, is called the Spanier group with respect to $\U$, denoted by $\pi(\U,x)$ \cite{S, R}.
\begin{definition}\cite{S, R}
The Spanier group of the space $X$ which we denote it by $\psp$, is defined as follows:
\[
\psp=\bigcap\limits_{open\ covers\ \U}\pi(\U,x).
\]
\end{definition}
 The authors \cite{P4} introduce Spanier spaces which are spaces such that their Spanier groups are equal to their fundamental groups. Also, the authors prove that for a connected and locally path connected space $X$, $\ov{\{[e_x]\}}\sub\psp$. Hence, for an $(n+1)$-$tuple$ of spaces $(X,A_1,A_2,...,A_n)$ with the assumption $(\clubsuit)$, where $X$ is simply connected, we have
  $$\pi_1^{qtop}(X/(A_1,...,A_n),*)=\ov{p_*\pi_1(X,x)}=\ov{\{[e_x]\}}\sub\pi_1^{sp}(X/(A_1,...,A_n),*).$$
Clearly simply connected spaces are Spanier spaces which we can call them trivial Spanier spaces. It is interesting for the authors to obtain some ways to construct nontrivial Spanier spaces. The following result which is an immediate consequence of the above argument gives a way to construct some Spanier spaces from simply connected spaces.
\begin{theorem}
For an $(n+1)$-$tuple$ of spaces $(X,A_1,A_2,...,A_n)$ with the assumption $(\clubsuit)$, if $X$ is simply connected, then $X/(A_1,A_2,...,A_n)$ is a Spanier space.
\end{theorem}
In the following example, we show that there exists a simply connected, locally path connected metric space $X$ with a closed path connected subspace $A $ such that $X/A$ is not simply connected and by Theorem 4.1 $\pi_1^{qtop}(X/A,*)$ is an indiscrete topological group. Hence $X/A$ is a nontrivial Spanier space.
\begin{example}
Using the definitions of Example 3.16, let $CHE_o$ and $CHE_e$ be cones over $HE_o$ and $HE_e$ with height $\fr{1}{2}$ and let $X=CHE_o\cup CHE_e\cup A$. By the van Kampen theorem, $X$ is simply connected, but $X/A$ is not simply connected (see \cite{Gr}). Hence $X/A$ is a nontrivial Spanier space.
\end{example}

\subsection*{Acknowledgements}
The authors would like to thank the referee for the valuable comments and suggestions which improved the manuscript and made it more readable.

\end{document}